\documentclass[12pt]{amsart}

\usepackage[draft]{say}

\usepackage{graphicx}
\usepackage{amssymb}
\usepackage{amsmath, amscd}
\usepackage{dbnsymb}
\usepackage[hidelinks]{hyperref}

\usepackage{physics}
\usepackage{amsmath}
\usepackage{tikz, tikz-cd}
\usepackage{mathdots}
\usepackage{yhmath}
\usepackage{cancel}
\usepackage{color}
\usepackage{siunitx}
\usepackage{array}
\usepackage{multirow}
\usepackage{amssymb}
\usepackage{gensymb}
\usepackage{tabularx}
\usepackage{extarrows}
\usepackage{booktabs}
\usetikzlibrary{fadings}
\usetikzlibrary{patterns}
\usetikzlibrary{shadows.blur}
\usetikzlibrary{shapes}

\newtheorem{thm}{Theorem}[section]

\newtheorem{theorem}[thm]{Theorem}
\newtheorem{proposition}[thm]{Proposition}
\newtheorem{corollary}[thm]{Corollary}
\newtheorem{lemma}[thm]{Lemma}

\theoremstyle{definition}
\newtheorem{definition}[thm]{Definition}

\newtheorem{construction}[thm]{Construction}

\theoremstyle{remark}
\newtheorem{remark}[thm]{Remark}

\newcommand{\C}{{\mathbb{C}}}
\newcommand{\Q}{{\mathbb{Q}}}
\newcommand{\Z}{{\mathbb{Z}}}

\usepackage[margin=1in,marginparwidth=0.8in, marginparsep=0.1in]{geometry}

\usepackage{epigraph}

\usepackage{wrapfig, framed, caption}

\begin{document}

\title[A skein-valued lift of the unidirectional $A_n$-quiver $Q$-dilogarithm relations]{A skein-valued lift of the unidirectional $A_n$-quiver $Q$-dilogarithm relations} 
\author{Matthias Scharitzer}
\address{Centre for Quantum Mathematics, SDU, Campusvej 55, 5230 Odense M, Denmark}
\email{matthias.scharitzer@gmail.com} 

\maketitle

\begin{abstract} 
  We show that the q-dilogarithm identities associated to the unidirectional $A_n$-quiver lift to the HOMFLYPT-skein algebra of a genus $n$ handlebody. 
\end{abstract}

\thispagestyle{empty} 

\section{Introduction}

\subsection{$Q$-Dilogarithm relations in quantum tori, following \cite[Section 1]{Keller}}

In this paper, we will consider generalizations of the celebrated pentagon relation:

\begin{theorem}(Pentagon relation \cite{schutzenberger,Faddeev1993,Faddeev1994}, as cited by \cite[Theorem 1.2.]{Keller})
   Consider the algebra of formal power series $\mathfrak{A}_{A_2}$ in two non-commutative variables $x_1,x_2$ with coeffcients in $\Q[q^{-\frac{1}{2}}][[q^{\frac{1}{2}}]]$ subject to the $q$-commutation relation $x_1x_2=qx_2 x_1$. Then inside $\mathfrak{A}_{A_2}$ the following relation holds:

    \begin{equation*}
        \mathcal{E}_q(x_{1})\mathcal{E}_q(x_{2}) = \mathcal{E}_q(x_{2})\mathcal{E}_q(q^{\frac{1}{2}}x_1x_2)\mathcal{E}_q(x_{1})
    \end{equation*}

    where $\mathcal{E}_q$ is the (exponential of) the $q$-dilogarithm, explicitly given by:

    \begin{equation}
    \mathcal{E}_q(x_{\alpha}) = \sum \frac{1}{\Pi_{0\leq i\leq j}q^{\frac{i}{2}} - q^{-\frac{i}{2}} } x_\alpha^j
\end{equation}
\end{theorem}

More generally, let $Q$ be an $A_n$-quiver. This quiver has no cycles, so one can consider some choice of order of the vertices $v_1<\dots<v_n$ which obey $v_i < v_j$ if there is an arrow from $v_i$ to $v_j$. Associated to this quiver one can define an antisymmetric form on the lattice $\mathbb{N}^n$ by setting $<e_i,e_j>$ equal to $\frac{1}{2}$ if there is an arrow from $v_i$ to $v_j$ and $0$ if there are no arrows between $v_i$ and $v_j$ where $e_i,e_j$ denote the $i$-th, respectively $j$-th unit vector. Starting from this one can consider the quantum torus $\mathfrak{A}_Q$ which is the algebra of power series in the variables $x_\alpha$ where $\alpha \in \mathbb{N}^n$ with coefficients in the ring $\Q[q^{-\frac{1}{2}}][[q^{\frac{1}{2}}]]$ subject to the relation $x_{\alpha}x_{\beta}=q^{<\alpha,\beta>}x_{\alpha+\beta}$.

The more general version of the pentagon relation for an arbitrary $A_n$-quiver, is obtained in the following way: Consider the set $\{V_1,\dots,V_N\}$ of indecomposable $Q$-representations. We endow this set with an order $V_1<\dots<V_n$ which is induced by $V_i<V_j$ if $Hom_Q(V_i,V_j) \neq 0$. Furthermore, to each indecomposable representation, we can associate its dimension vector $d_i \in \mathbb{N}^n$, where we identify $e_i$ with the simple representation which is $0$ on all vertices except above $v_i$ where it is $\C$. Then we have the following $Q$-relations:

\begin{theorem}(\cite{Reineke}, as cited by \cite[Corollary 1.7.]{Keller})
If $Q$ is an $A_n$-quiver and $v_1<\dots<v_n$ is an order on the vertices on $Q$ and $V_1<\dots<V_N$ is an order on the indecomposable representations $V_1,\dots,V_N$ such that $V_i<V_j$ if $Hom_Q(V_i,V_j) \neq 0$. Let $d_i$ be the dimension vectors of the indecomposable representations then we have the following relations:

\begin{equation} \label{Q quantum equation}
    \mathcal{E}_q(x_{1})\dots \mathcal{E}_q(x_{n}) = \mathcal{E}_q(x_{d_N})\dots\mathcal{E}_q(x_{d_1})
\end{equation}
    
\end{theorem}

The theorem above is true for a more general class of quivers for more details, see \cite{Reineke}.

\subsection{HOMFLYPT and linking skein algebras} \label{Section 1.2.}

\begin{figure}
    \begin{picture}(225,125)
    \put(0,0){\includegraphics[width=7cm]{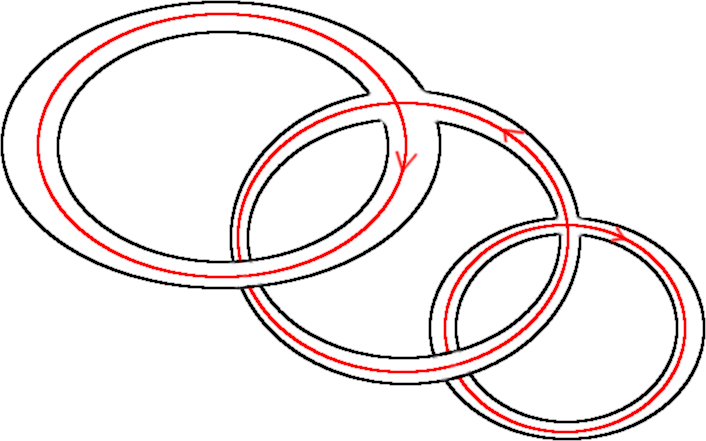}}

    \end{picture}
    \caption{The twisted disk $D_{n,tw}$ with $n=3$. The red oriented curves are, from left to right, $L_1$, $L_2$ and $L_3$.}  
    \label{Fig:Twisted Disk}
\end{figure}

From now on, let $Q$ be the unidirectional $A_n$ quiver. Then the quantum torus $\mathfrak{A}_Q$ appears geometrically in the following way: To a $3$-manifold $M$, we can associate the following two modules: Consider the module freely generated over $\Q [q^{-\frac{1}{2}},a^{\pm 1},c_{\lambda,\mu}^{-1}][[q^{\frac{1}{2}}]]$ by isotopy classes of framed links $K \subset M$, where $c_{\lambda,\mu}$ are eigenvalues of the sliding operator, see Section \ref{Subsection:Sliding Operator} for details.\footnote{Usually the HOMFLYPT module is defined over $\Z[a^{\pm 1},z^{\pm 1}]$, where $z=q^\frac{1}{2}-q^{-\frac{1}{2}}$ However for technical purposes, we need to work over this completed coefficient ring.} Fix an orientation of $M$, then we can consider the set of relations:

\begin{eqnarray*}
    \overcrossing - \undercrossing &=& (q^\frac{1}{2}-q^{-\frac{1}{2}}) \,\, \smoothing \\
    \bigcirc &=& \frac{a-a^{-1}}{q^\frac{1}{2}-q^{-\frac{1}{2}}}
\end{eqnarray*}

\begin{figure}
    \begin{picture}(150,40)
    \put(0,0){\includegraphics[width=5cm]{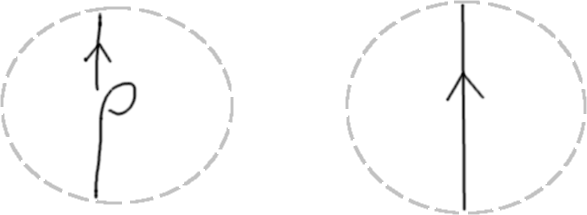}}

    \end{picture}
    \caption{If $L$ and $L'$ are isotopic but their framing is different then one can make the framing align along the whole knot except in a small neighborhood. Assume that, we endow this neighborhood with the framing of $L$. Then if $L'$ is obtanied from $L$ by a full twist, we need to add a loop as on the left-hand side, so that the ambient framing and the one along $L'$ agree.}  
    \label{Fig:Twist-identity}
\end{figure}

and $[L]=a[L']$ if the underlying knots of $L$ and $L'$ agree but the framing of $L'$ is obtained from $L$ by a full twist, see Figure \ref{Fig:Twist-identity}. The module obtained in this way is called the HOMFLYPT-skein module $Sk(M)$ of $M$. One can take a further quotient, given by:

\begin{eqnarray*}
    \overcrossing &=& q^{1/2} \,\, \smoothing \\ q^{\frac{1}{2}}&=&a
\end{eqnarray*}

The so obtained skein-module is sometimes called the linking skein $Lk(M)$ or also the $U(1)$-skein. In the cases of interested $M$ is either a genus $n$ handlebody $\mathcal{H}_n$ or a genus $n$ surface $\Sigma_n$ times interval $I$. Then the linking skein $Lk(M)$ is generated by one representative from each $H_1(M)$ class. 

Furthermore, $H_1(M)$ defines a grading of the skein-module. After fixing a strictly convex cone $\mathcal{C}=\mathbb{N}^k \subset H_1(M)$, we can consider the submodules $Sk_\mathcal{C}(M),Lk_\mathcal{C}(M)$ which is obtained by taking infinite sums of links with the restriction that any link appearing must have homology class in $\mathcal{C}$ and the sum may have only finitely many knots in any given homology class.

On $M=\mathcal{H}_n$ we can define an algebra operation in the following way: We can identify $\mathcal{H}_n$ with the product of the $n$-holed twisted disk $D_{n,tw}$, see Figure \ref{Fig:Twisted Disk}, and an interval $I$ \footnote{Strictly speaking, the smooth structures do not agree as smoothness detects the corners along the common boundary of $D_{n,tw}$ and $I$ whereas $\mathcal{H}_n$ has no corners. However, any knot representative can be slightly isotoped to be disjoint from this corner and thus there is no confusion about the smooth inclusion.}. Then the stacking operation in the interval direction, defines an algebra structure on $Lk_\mathcal{C}(\mathcal{H}_n)$. The choice of curves $L_1,\dots,L_n$, see Figure \ref{Fig:Twisted Disk}, defines a basis of $H_1(\mathcal{H}_n)$. Let $\mathcal{C}$ be the strictly convex cone generated by the $L_i$, then the algebra structure on $Lk_\mathcal{C}(\mathcal{H}_n)$ endowed with the (interval) stacking operation thus defines an identification of $Lk_\mathcal{C}(\mathcal{H}_n)$ with $\mathfrak{A}_Q$.\footnote{This identification is not quite the obvious one as $L_2L_1=q^{1}L_1L_2$. However, this choice is more convenient for notation later on. So, for this isomorphism we have to identify $x_i$ with $L_{n-i+1}$ and reverse the order of multiplication.} \footnote{One may note that the additional elements, $c_{\mu_1,\mu_2}$ are finite polynomials in $q^{-\frac{1}{2}},q^\frac{1}{2}$ so they are already invertible in $\Q[ q^{-\frac{1}{2}}][[ q^{\frac{1}{2}}]]$, so both $Lk_\mathcal{C}(\mathcal{H}_n)$ and $\mathfrak{A}_Q$ are generated over the same coeffcient ring.}
\subsection{Statement of the main result} \label{Subsection:Statement of the main result}

The indecomposable representations of $Q$ are given by representations $V_{[i,j]}$ indexed by intervals $[i,j]$ with $1 \leq i \leq j \leq n$ so that $V_{[i,j]}$ assigns $\C$ to the $i$-th through the $j$-th vertex with non-zero maps in-between and $0$ to all other vertices. So any dimension vector of an indecomposable representation is given by $d_{[i,j]}$ is given by an adjacent sequence of $1$'s. Similarly, we define for any interval $[i,j]$ $L_{[i,j]}$ to be the successive Dehn-twist of $L_i$ along $L_{i+1}$ through $L_j$.

Similarly, as for quantum tori, we can also define a skein-valued lift $\mathbb{E}$ of the $q$-dilogarithm \cite[Theorem 1.1.]{unknot} where $\gamma=a^{-1}$. We will not write it out explicitly here, as none of our results depend on the exact representation. See \cite[Section 2]{Scharitzer-Shende-2} for an exposition, on why this element is indeed a skein-valued lift of the quantum dilogarithm. We only note that in the expansion of $\mathbb{E}$ in \cite{unknot}, one needs to fix a solid torus inside $M$. So, if we write $\mathbb{E}(L)$ where $L \subset M$ is a framed one-component knot, we mean the insertion of the knots of $\mathbb{E}$ inside this torus. In addition, $\mathbb{E}(L)$ consists of finitely many knots of homology class $k|L|$ for each $k\geq 0$. So if $|L|\in \mathcal{C}\setminus \{0\}$ then $\mathbb{E}(L) \in Sk_\mathcal{C}(M)$. 

\begin{theorem}(Skein-valued lift of the $Q$-relations) \label{Main theorem}
    Let $D_{n,tw}$ be the twisted $n$-holed disk from Figure \ref{Fig:Twisted Disk} and $L_1,\dots,L_n$ be the curves indicated there. Furthermore, fix an order $\alpha_1<\dots<\alpha_N$ on $[1,1],[1,2],\dots,[1,n],\dots,[n,n]$ such that the irreducible representations $V_{\alpha_i}$ obey that $V_{\alpha_j}<V_{\alpha_i}$ if $Hom_Q(V_{\alpha_i},V_{\alpha_j}) \neq 0$. Then the following relation holds in $Sk_\mathcal{C}(\mathcal{H}_n)$:

    \begin{equation} \label{Main equation introduction}
        \mathbb{E}(L_{1})\dots \mathbb{E}(L_{n}) = \mathbb{E}(L_{\alpha_N})\dots\mathbb{E}(L_{\alpha_1})
    \end{equation}

    where $\mathcal{C}$ is the cone spanned by the $L_1,\dots,L_n$.
\end{theorem}

     The concatenation of the map $Sk_\mathcal{C}(\mathcal{H}_n) \rightarrow Lk_\mathcal{C}(\mathcal{H}_n)$ with the identification of $Lk_\mathcal{C}(\mathcal{H}_n) \cong \mathfrak{A}_Q$ sends Equation \ref{Main equation introduction} to Equation \ref{Q quantum equation}.

Proofs for these relations in the case of the $A_2$-quiver appeared in earlier works by Nakamura \cite{Nakamura} and Hu \cite{Hu}. Their proofs work more directly with calculations of skeins whereas this manuscript is based on the techniques of \cite{Scharitzer-Shende,Scharitzer-Shende-2,HSZ}. The upshot of using these papers is that they repackage many elaborate calculations of skein-relations into an easy to use graphical calculus of cubic planar graphs which allows us to keep the actual computation of skeins to a minimum.

\vspace{2mm}

{\bf Acknowledgements.}  I want to especially thank Vivek Shende for suggesting this question and mentoring during the preparation of this manuscript. I would also like to thank Gard Olav Helle for helpful discussions. This work is supported by Novo Nordisk Foundation grant NNF20OC0066298.

\section{Characterizations of the order of indecomposable representations}

In this section, we will state two lemmas which help us characterize the order described by Keller \cite{Keller} more clearly.

\begin{lemma} \label{Quiver relation order}
    The condition $Hom_{A_n}(V_{\alpha},V_{\beta}) \neq 0$ implies $\beta \leq \alpha$ on the set of all intervalls $[i,j]$ where $1 \leq i \leq j \leq n$ is equivalent to the condition that any order of that set must obey $[i,j]<[i,j+1]$ and $[i,j]<[i+1,j]$.
    \end{lemma}

\begin{proof}
   Let $V_{[i,j]}$ and $V_{[i',j']}$ be two indecomposable representations of the given type. We claim that $Hom_Q(V_{[i,j]},V_{[i',j']})=0$ if and only if $i<i'$ or $j<j'$ or the intervals are disjoint. This directly implies that $[i,j] \leq [i',j']$ if $i \leq i'$ and $j=j'$ or $i=i'$ and $j \leq j'$. Since $\alpha_1,\dots,\alpha_N$ exhausts all such intervals this immediately simplifies to the claimed order conditions. 
   
   We prove the claim by studying the appearing commutative squares. Let $V_{a,b}$ for $a,b=1,2$ be the four corners of the commutative square and we have oriented arrows from $V_{1,1}$ to $V_{1,2},V_{2,1}$ and from $V_{1,2},V_{2,1}$ to $V_{2,2}$. The corners are either labelled by $0$ or $\C$. By assumption, we consider only indecomposable representations, so the arrows between $V_{1,1}$ and $V_{1,2}$, respectively $V_{2,1}$ and $V_{2,2}$ are non-zero iff both are labelled by $\C$. We will argue that the above restriction on $Hom_Q$ comes from studying those commutative squares which have at most $1$ vertex labelled with $0$. The other commutative squares either impose no restriction on the vertical maps or all involved maps are already $0$. So, consider the square where all vertices are labelled with $\C$. Then, one sees that $\mu_{lower}f_1=\mu_{upper}f_{2}$ so $f_{1}$ is non-zero iff $f_{2}$ is non-zero. This implies that if there is a sequence of commutative squares where all corners are $\C$, either all vertical maps are non-zero or all are $0$. Indeed, all squares with only $\C$ appearing on the corners are connected for $Hom_Q(V_{[i,j]},V_{[i',j']})$.
   
   Next, consider the square where $V_{1,2}$ is $0$. Then $f_2$ is automatically $0$ which implies $\mu_{lower}f_1=0$ and thus $f_1$ must be $0$. This square appears in $Hom_Q(V_{[i,j]},V_{[i',j']})$ if $j<j'$. Next, assume that $V_{2,1}$ is $0$. Then automatically $f_1=0$ and by commutativity $\mu_{upper}f_2=0$ thus $f_2$ is $0$. This square appears in $Hom_Q(V_{[i,j]},V_{[i',j']})$ if $i< i'$. One directly verifies that in the other two squares there is also at most one non-trivial vertical map, however it is pre-composed by the $0$-map. Thus the commutativity condition is void and thus the vertical map is arbitrary. So, we see that if either $j<j'$ or $i<i'$ then one vertical map between two copies of $\C$ is $0$. Since, all commutative squares with $\C$ on the corners are connected to one another, this shows that for $[i,j]\cap [i',j']\neq \emptyset$ $Hom_Q(V_{[i,j]},V_{[i',j']})=0$ iff $i<i'$ or $j<j'$. The other $Hom$-sets are obviously $0$ as well, if the sets are disjoint, as all vertical arrows either have $0$ in the domain or target.
\end{proof}

\begin{lemma} \label{Quiver-equivalent formulation of sequence}
    Let $A_1,\dots,A_N$ be some enumeration of $[1,1],[1,2],[2,2],\dots,[1,n],\dots,[n,n]$ which obeys that $[i,j]<[i,j+1]$ and $[i,j]<[i+1,j]$. This sequence can be equivalently described by a sequence $a_0,\dots,a_n$ of $n$-tuples $a_k=(i_1,\dots,i_n)$ which fulfill the following:
    \begin{enumerate}
        \item $i_j\leq j$ for all $j=1,\dots,n$;
        \item $a_0=(0,0,\dots,0)$;
        \item Let $M_k$ be the maximal number such that the first $M_k$ entries of $a_k$ agree with the first $M_k$ entries of $(1,\dots,n)$. Then the subsequence of $a_k$ obtained from deleting the first $M_k-1$ entries is non-increasing;
        \item $a_{k}$ is obtained from $a_{k-1}$ by increasing the first entry of any maximal constant subsequence of the last $n-M_{k-1}$ entries of $a_{k-1}$ by $1$.
    \end{enumerate}
\end{lemma}

\begin{proof}
    Let $A_1,\dots,A_k$ be the initial segment of such an enumeration then we define the $j$-th entry of $a_k$ by the largest $i$ such that $[i,j]$ appears in $A_1,\dots,A_k$, otherwise set $i_j$ to be $0$. This directly implies the property $(1)$ as $[i,j]$ only exists for $i\leq j$. Property $(2)$ is necessarily true as $a_0$ is defined on the empty list and so no $[i,j]$ appeared yet. Property $(3)$ follows directly from observing that the $(j+1)$-th entry $i_{j+1}$ of $a_k$ being larger than $i_{j}$ the $j$-th entry would imply that $[i_{j+1},j+1]$ appeared but $[i_{j},j]$ did not yet appear which is impossible as $A_1,\dots,A_N$ exhausts all intervals and $ [i_{j},j] < [i_{j}+1,j] < [i_{j}+1,j+1] < \dots <[i_{j+1}, j+1]$ and thus $[i_{j}+1,j]$ could not be fit into $A_1,\dots,A_N$. This condition is void if $[i_{j}+1,j]$ does not exist, i.e. $i_{j}=j$. So this condition does not apply to the initial segment which is thus allowed to be strictly increasing. Property $(4)$ follows similarly as the sequence would otherwise either violate property $(3)$ or stagnate, meaning that $A_{k+1}$ is $[i,j]$ but $[i_{j},j]$, for $i<i_j$, already appeared before which is impossible as $[i,j]<\dots <[i_j,j]$.

    On the other hand, let $a_0,\dots,a_N$ be a sequence as described above. By property $(4)$, $a_1=(1,0,\dots,0)$, so set $A_1=[1,1]$ and for $k>1$ define $A_k$ to be $[i_j,j]$, where $j$ is the unique entry of $a_k$ from Property $4$ and $i_j$ is the entry of $a_k$. This cannot produce an element $[i,j]$ with $i>j$ as $i=j$ is only true in $a_{k-1}$ in the first $M_{k-1}$ entries and these remain the same by Property $(4)$. Since the list has $N+1$ entries which by property $(4)$ are necessarily different from one another this list must exhaust all allowed intervals. Assume that $A_{k}$ is the interval $[i,j]$ then we claim that both $[i-1,j]$ and $[i,j-1]$ must appear in the list earlier. Since the $j$-th entry can only ever be raised by $1$ and starts at $0$, some earlier $A_{k'}<A_k$ must have been the interval $[i-1,j]$. Similarly, if $A_k=[i,j]$ then by Property $(4)$ the $(j-1)$-th entry in $a_{k-1}$ must at least be $i$, as otherwise either $a_{k-1}$ failed to be non-increasing or the $j$-th entry was not the first entry of the maximal constant subsequence with value $i-1$. This completes the argument.
\end{proof}

\section{Background and results from \cite{Scharitzer-Shende, Scharitzer-Shende-2,HSZ}}

In this section, we recall some of the basic set-up from \cite{Scharitzer-Shende, Scharitzer-Shende-2,HSZ} that we will use later on. These works are very different in nature, while \cite{Scharitzer-Shende, Scharitzer-Shende-2} were motivated mostly by geometry, the second group \cite{HSZ} approached these questions mostly from the algebraic direction. Even though, for our purposes there is little difference between these two works, we will more closely follow the notation of \cite{Scharitzer-Shende, Scharitzer-Shende-2}.

\subsection{Boundary inclusion and further completions} \label{Section 3.1.}

In Subsection \ref{Section 1.2.}, we defined an algebra structure on the skein-module $Sk(\mathcal{H}_n)$ by identifying $\mathcal{H}_n \cong D_{n,tw} \times I$. Similarly, we define an algebra action on $Sk(\Sigma_n \times I)$. After fixing an identification of $\Sigma_n \cong \partial \mathcal{H}_n$, we also obtain a left action of $Sk(\Sigma_n \times I)$ on $Sk(\mathcal{H}_n)$, this last action we term "boundary inclusion". One readily observes, that once one fixes a subset $D_{n,tw} \subset \Sigma_n$ so that the identification of $\Sigma_n \cong \partial \mathcal{H}_n$ matches with the identification $\mathcal{H}_n \cong D_{n,tw} \times I$, then the stacking operation of $D_{n,tw} \times I$ filters through the inclusions of $D_{n,tw} \times I \subset \Sigma \times I \subset \mathcal{H}_n$ and so the action of $Sk(\Sigma_n \times I) \subset Sk(\mathcal{H}_n)$ when restricted to $Sk(D_{n,tw}\times I) \subset Sk(\Sigma_n \times I)$ coincides with the left stacking action of Subsection \ref{Section 1.2.}.

Let $\mathcal{C}$ be a strictly convex cone in $H_1(\mathcal{H}_n)$ and $\mathcal{C}'$ be a strictly convex cone in $H_1(\Sigma_n)$. Assume that the boundary inclusion $\iota$ fulfills $\iota^*\mathcal{C}' \subset \mathcal{C}$ and $\mathcal{C}'\cap ker(\iota^*(H_1(\Sigma_n))= \{0\}$ then the boundary inclusion map $Sk_{\mathcal{C}'}(\Sigma_n \times I) \rightarrow Sk_{\mathcal{C}}(\mathcal{H}_n)$ is well-defined. However, while our elements $\mathbb{E}$ will always be elements of such completions, we will need to consider products with finitely many knots which live outside such a strictly convex cone. So let $h_1,\dots,h_l \in H_1(M)$ be a finite list of homology class elements. Then denote by $Sk_{\mathcal{C},h_1,\dots,h_l}(M)$ the skein module which allows infinite sums of elements of $Sk(M)$ but the sum is finite in each homology class and is empty outside the union of the affine strictly convex cones $h_i+\mathcal{C}$. Fix $\mathcal{C}$ and $\mathcal{C}'$ strictly convex cones in $H_1(M)$, respectively $H_1(\Sigma)$, finitely many homology classes $h_1,\dots,h_l \in H_1(M)$, respectively $h'_1\dots,h'_k \in H_1(\Sigma_n)$ and an identification of $\Sigma_n \cong \partial M$. If the boundary inclusion $\iota$ preserves the cones, i.e. $\mathcal{C}'\cap ker(\iota)=\{0\}$ and $\iota^* (\bigcup h'_i+\mathcal{C}') \subset \bigcup (h_i+\mathcal{C})$ then there is an induced boundary inclusion map $Sk_{\mathcal{C' },h'_1,\dots,h'_k}(\Sigma_n \times I) \rightarrow Sk_{\mathcal{C},h_1,\dots,h_l}(M)$. We will often abuse notation and equate the modules $Sk_{\mathcal{C}}(M)$ and $Sk_{\mathcal{C},h_1,\dots,h_l}(M)$. However, any element of the type $\mathbb{E}$ which appears in this paper, will live in the submodule $Sk_{\mathcal{C}}(M)$.

\subsection{Twisted skein modules} \label{twisted skein module section}

\begin{definition}
    Let $M^3$ be a compact oriented $3$-manifold with boundary, $\Sigma$ its boundary and $\mathfrak{p}$ an oriented $1$-chain which is closed relative to the boundary . Then we define the $\mathfrak{p}$-twisted skein-module $Sk(M^3,\mathfrak{p})$ as the $\Q [a^{\pm 1},c_{\lambda,\mu}^{-1},q^{-\frac{1}{2}}][[q^{ \frac{1}{2}}]]$-module whose generators are all framed links inside $M^3 \setminus \mathfrak{p}$ modulo the HOMFLYP-skein relations and the (signed) framing line relation which is given by:
    \begin{equation*}
        \overcrossing = -a^{-1} \undercrossing
    \end{equation*}
    where the line pointing to the left is a section of $\mathfrak{p}$ and the right pointing line is a section of the knot. We call $\mathfrak{p}$ the framing lines.
\end{definition}

Assume that $\mathfrak{p}$ is exact relative to the boundary and $D$ a $2$-chain bounding $\mathfrak{p}$ then we define the $D$-untwisting map by:

\begin{eqnarray*}
    Sk(M^3,\mathfrak{p}) &\rightarrow& Sk(M^3) \\
    \mathbf{}[L] &\mapsto& (-a)^{D \cap L}[L]
\end{eqnarray*}

In this paper, all twisted skein modules have $M=\Sigma_n \times I$ as underlying manifold. So after fixing a finite number of oriented points $\mathfrak{p} \subset \Sigma_n$ there is an evident analogue of the boundary inclusion/stacking operation on $Sk(\Sigma\times I,\mathfrak{p}\times I)$. Similarly, if we let $\mathcal{C}$ be a strictly positive cone, we can define the $(\mathfrak{p}\times I)$-twisted skein-module $Sk_\mathcal{C}(\Sigma_n \times I, \mathfrak{p}\times I)$.

\subsection{The sliding operator} \label{Subsection:Sliding Operator}

For this subsection, assume that $M^3$ is a handlebody and $\mathfrak{p}=\emptyset$. Assume that $c$ is a simple closed curve in the boundary $\Sigma$ of $M^3$ which becomes contractible after including it in $M^3$. Then we can define the sliding operator: $([c]-[\bigcirc])$. Intuitively, this operator measures how often any given link $L$ passes through a disk $D$ with $\partial D=c$ non-trivially. In fact, one can prove the following lemma:

\begin{lemma} \label{Sliding operator lemma}
    Let $\Psi \in Sk_\mathcal{C}(M^3)$ and $c$ be a simple closed curve in the boundary of $M^3$. If $[c]\Psi-[\bigcirc]\Psi=0$ then $\Psi \in Sk_\mathcal{C}(M^3; D)$ where $D$ is some disk bounding $c$ and $Sk_\mathcal{C}(M^3; D_1,\dots,D_k)$ denotes the submodule of  $Sk_\mathcal{C}(M^3)$ generated by all framed links which do not intersect the disks $D_1,\dots,D_k$.
\end{lemma}

The idea of this lemma goes back to \cite[Lemma 2.2]{Gilmer-Zhong-connectsum}. The version written above is a slight adaptation of the lemmas which appeared in \cite[Corollary 8.5]{Scharitzer-Shende-2} and \cite[Lemma 7.2]{HSZ}.

\begin{remark}
    As mentioned in the Introduction, see Section \ref{Section 1.2.}, we needed to adjoint the inverses of certain elements $c_{\mu_1,\mu_2}$ to the coefficient ring of our skein-modules: These are exactly the eigenvalues of the sliding operator $([c]-[\bigcirc])$ above and these are indexed by pairs of Young diagrams where $(\mu_1,\mu_2) \neq (\emptyset,\emptyset)$, for details see \cite{Gilmer-Zhong-connectsum}.
\end{remark}

\subsection{Skeins associated to cubic planar graphs}

Let $\Gamma\subset S^2$ be a (loop-free) cubic planar graph and fix a map $\pi$ from $\Sigma_g \rightarrow S^2$ to the double cover branched over the vertices of $\Gamma$. Each edge of $\Gamma$ lifts to a simple closed curve $E$ on $\Sigma_g$. Similarly, each face of $\Gamma$ lifts to two polygons which intersect each other at its bounding vertices. There is a coherent way to mark the lifts of the faces with $+$ and $-$ such that if two such polygons are separated by an edge then one is marked with $+$ and one with $-$, see for instance the wavefront description \cite{Treumann-Zaslow}. This induces an orientation on the set $\{E\}$ of lifts of the edges of $\Gamma$: Each half of $E$ (separated by the vertices of the corresponding edges) is a boundary curve of a face marked with $+$. This induces an orientation on $E$ as part of the boundary of $f$. This assignment is coherent, as the vertices where different faces meet are trivalent and thus the double cover is $6$-valent. As the faces marked with $+$ and $-$ have to alternate along edges, this means that both lifts of adjacent faces marked with $+$ have to remain on the same side of $E$.

On planar trivalent graphs there is a well-defined operation called the graph mutation. Usually this is depicted as in the first two pictures of Figure \ref{Fig:Mutation}, where one cannot uniquely identify the vertices before and after mutation. In particular, there are two such ways. One that we call the positive and the other the negative mutation. If we fix the identification of the vertices (and faces) as indicated in the second and third picture of Figure \ref{Fig:Mutation} then we obtain the positive mutation.

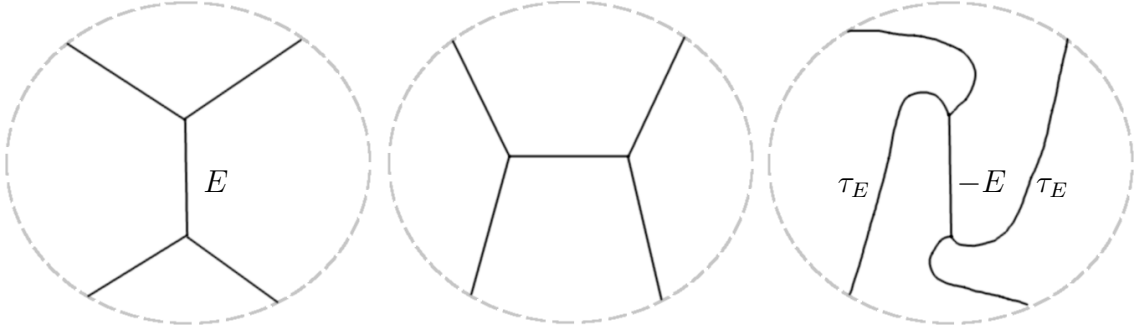
\begin{figure}
    \begin{picture}(550,125)
    \put(0,0){\includegraphics[width=15cm]{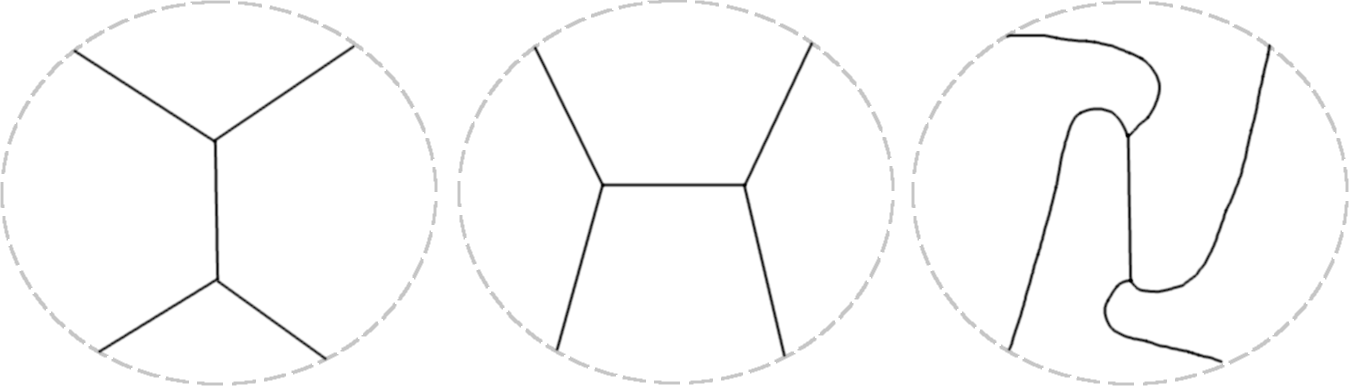}}
    \put(75,50){$E$}
    \put(360,50){$-E$}
    \put(315,50){$\tau_E$}
    \put(390,50){$\tau_E$}

    \end{picture}
    \caption{On the left the neighborhood of an edge $E$ in a trivalent graph. In the middle, the graph obtained by "graph mutation" at this edge. Once we fix one identification of the vertices in the left and middle pictures, we obtain the right picture, which is the so-called positige mutation. The labels in the first and third picture indicate the change of homotopy class. The homotopy class of the double lift of the central edge is reversed while those of the edges labeled with $\tau_E$ are Dehn-twisted along $E$.}  
    \label{Fig:Mutation}
\end{figure}

\begin{definition}(\cite[Lemma 9.1.]{Scharitzer-Shende-2},\cite[Definition 5.7.]{HSZ})
    Let $\Gamma$ be a (loop-free) cubic planar graph and $\pi:\Sigma_n \rightarrow S^2$ a branched double cover and $E$ an edge of $\Gamma$ which does not bound a bigon. Then, we define $\Gamma_E$ the graph obtained from $\Gamma$ by positive mutation along $E$ to be the graph which agrees with $\Gamma$ outside a neighborhood of $E$ and inside a neighborhood of $E$ the graph $\Gamma$ and $\Gamma_E$ differ as indicated in the first and third picture of Figure \ref{Fig:Mutation}. 
\end{definition}

In this work, we will only consider positive mutations and we have to record the relative change of homotopy classes before and after positive mutation. The change of homotopy class after fixing $\pi:\Sigma_n\rightarrow S^2$ and identifying edges is depicted in Figure \ref{Fig:Mutation}.

\begin{definition} (\cite[Definition 2.4.]{Scharitzer-Shende}/\cite[Equation 5.1.1./Example 5.1.]{HSZ}) \label{Face relation-definition}
    Let $\Gamma$ be a cubic planar graph, fix a cover $\pi:\Sigma_n \rightarrow S^2$ and fix a point $p_f$ in the interior of each face. Define $\mathfrak{p}$ to be the set of preimages of all $p_f$ oriented such that the point in the face marked with $+$ is positive and the point in the face marked with $-$ is negative. Then for each face $f$ and each vertex $v$ adjacent to $f$, we can define an associated element in $Sk(\Sigma_n \times I, \mathfrak{p}\times I)$:

     $$\mathcal{A}_{\Gamma,f,v}=a^{-1}[\bigcirc] + \sum_{1 \leq k \leq n-1}[l_{v}]$$

     where $n$ is the number of vertices adjacent to $f$ and $l_k$ is a knot fulfilling the following properties:

     \begin{enumerate}
         \item $l_k$ is the double lift of a path in $f$ from $v$ to $v_k$ the $k$-th vertex from $v$ counted counterclockwise which stays in a neighborhood of $\partial f$ such that the point $p_f$ stays outside this region.
         \item $l_k$ is framed by the positive $I$ vectorfield.
     \end{enumerate}
\end{definition}

\begin{theorem}(\cite[Proposition A.5.]{Scharitzer-Shende-2}/\cite[Theorem 5.8.]{HSZ}) \label{Tangle conjugation equation}
Let $\Gamma$ be a cubic planar graph with branched lift $\pi$, E the double lift of an edge of $\Gamma$, a strictly convex cone $\mathcal{C} \subset H_1(\Sigma_n)$ such that $E \in \mathcal{C}$ and $E$ is not contractible. Let $f$ be a face of $\Gamma$ and $v$ be a vertex of $f$ which is not adjacent to $E$. Then we have the following relation:

$$\mathcal{A}_{\Gamma_E,f,v}\mathbb{E}(E) = \mathbb{E}(E) \mathcal{A}_{\Gamma,f,v}$$

where $\Gamma_E$ is the graph obtained by positively mutating $\Gamma$ along $E$.
\end{theorem}

\begin{corollary}(\cite[Corollary A.6.]{Scharitzer-Shende-2}) \label{Mutation sequence corollary}
    Let $\Gamma_0$ be a cubic planar graph with branched lift $\pi$, $\mathcal{C} \subset H_1(\Sigma_n)$ be a strictly convex cone and $E_1,\dots,E_k$ be a sequences of simple closed curves on $\Sigma_n$ such that:
    
    \begin{enumerate}
        \item $E_i \in \mathcal{C}$;
        \item $E_i$ is not contractible  for $i=1,\dots,k$;
        \item $E_i$ is the oriented double lift of an edge of $\Gamma_{i-1}$ where $\Gamma_0,\Gamma_1,\dots,\Gamma_k$ are inductively defined so that $\Gamma_i$ is the positive mutation of $\Gamma_{i-1}$ along $E_i$. 
    \end{enumerate} Then we can define $\Psi_{E_1,\dots,E_k}:=\mathbf{E}(E_k)\dots\mathbf{E}(E_1) \in Sk_\mathcal{C}(\Sigma_n \times I, \mathfrak{p} \times I)$ which fulfills:

    $$\mathcal{A}_{\Gamma_k,f,v}\Psi_{E_1,\dots,E_k} = \Psi_{E_1,\dots,E_k}\mathcal{A}_{\Gamma,f,v} $$
    for each face $f$ of $\Gamma$ where $v$ is a vertex of $f$ which under the identification of vertices of the $\Gamma_0,\dots,\Gamma_n$ is not adjacent to any $E_i$.
\end{corollary}

\section{Constructions}

\subsection{The square and canoe graphs}

\begin{definition}
    For each $n\geq 2$, we define two cubic graphs $\Gamma^1_n,\Gamma^2_n \subset [-1,n+1] \times [-n,n]$ by the following set of vertices and edges:

    \begin{eqnarray*}
        V(\Gamma^1_n)&=&\{(a,b):a=0,\dots,n;b=-n,n\} \\
        E(\Gamma^1_n)&=&\{((a_1,b_1),(a_2,b_2)):b_1=b_2;a_1=a_2+1\} \\
        &\sqcup& \{((a_1,b_1),(a_2,b_2)):b_1\neq b_2;a_1=a_2\} \\
        &\sqcup& \{(0,n),(0,-n)\} \\
        V(\Gamma^2_n)&=&\{(a,b):a=0,\dots,n;b=n\}\\ 
        &\sqcup& \{(i,-n+i):i=0,\dots n\} \\ 
        E(\Gamma^2_n)&=&\{((a_1,b_1),(a_2,b_2)):b_1=b_2=n;a_1=a_2+1\} \\
        &\sqcup& \{((i,-n+i),(i+1,-n+i+1)): i=0,\dots,n-1 \} \\
        &\sqcup& \{((i,n),(i+1,-n+i+1):i=0, \dots n-1)\}\\
        &\sqcup& \{((0,n),(0,-n)) , ((n,0),(n,-n))\}
        \end{eqnarray*}

        We call $\Gamma^1_n$ the square graph and $\Gamma^2_n$ the canoe graph.
\end{definition}

In Figure \ref{Fig:Quiver-Graphs}, we have indicated these graphs for the case $n=2$. In the same Figure, we have put labels on $\Gamma^1_2$ which indicate the homology class expanded in some fixed basis $L_1,M_1,L_2,M_2$ of the associated branched double cover $\Sigma_n$. For each $\Gamma^1_n$, we will fix analogous homology labels on the edges which equal the homology classes of the edges in $\Sigma_n$. Note that, when considering mutation sequences between these graphs, we will always identify the vertices $(i,-n)$ in $\Gamma^1_n$ and $(i,i-n)$ in $\Gamma^2_n$ with one another. We could have chosen notations in such a way that graph mutations fix the absolute position in our Figures, however it is convenient to let the absolute $y$-coordinate recover some additional information. We will never mutate along an edge which ends on a vertex on the $(\cdot,n)$ line. So throughout we will fix notation in such a way that graph mutation identifies all vertices along the $(i,Y)$ half-line where $Y\leq 1$. 

\begin{figure}
    \begin{picture}(550,125)
    \put(0,0){\includegraphics[width=18cm]{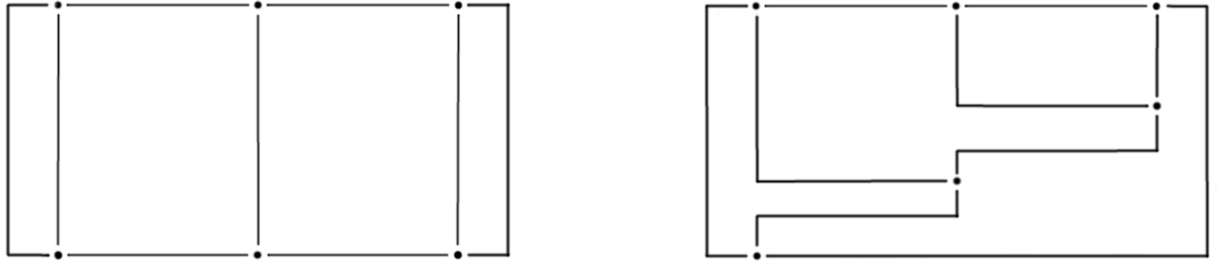}}
    \put(-22,50){$-M_1$}
    \put(30,50){$M_1$}
    \put(110,50){$M_2$}
    \put(135,50){$-M_1-M_2$}
    \put(215,50){$M_1+M_2$}
    \put(50,8){$L_1$}
    \put(120,8){$L_2$}
    \put(25,110){$-L_1-M_1-M_2 $}
    \put(120,110){$-L_2+M_1$}

    \end{picture}
    \caption{The graphs $\Gamma^1_n$ and $\Gamma^2_n$. The labels on the edges depict the homology classes of the lifts in a particular choice of basis of $H^1(\Lambda_{\Gamma^i_n})$}  
    \label{Fig:Quiver-Graphs}
\end{figure}
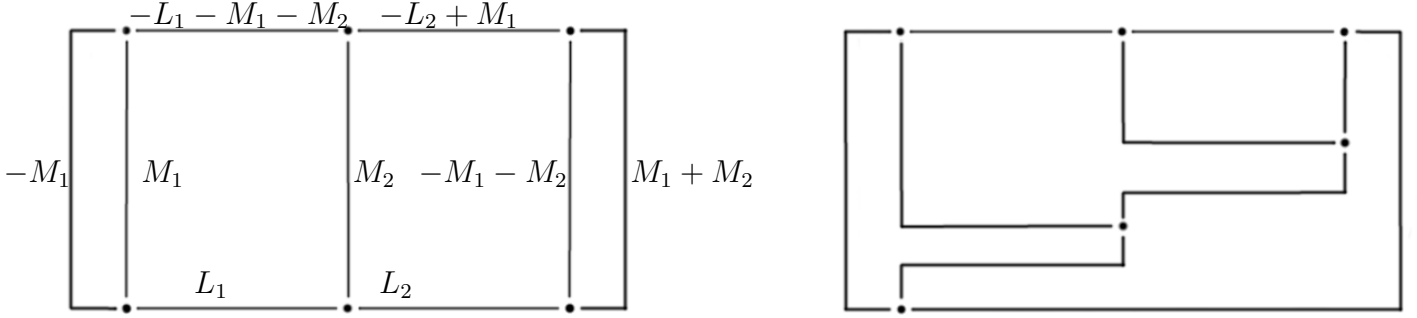

\subsection{A compatible filling of the square graph}

\begin{construction} \label{Construction of filling}
    Let $\Gamma^1_n$ be the square graph. Then there is a branched cover $\pi$ and a handle-body $\mathcal{H}_n$ with boundary $\Sigma_n$. This filling can be chosen to fulfill the following properties:

    \begin{enumerate}
        \item $H_1(\mathcal{H}_n)$ is generated by the inclusions of $L_i$;
        \item The inclusion $\Sigma_n \rightarrow \mathcal{H}_n$ contracts the edges with homology classes $M_i$;
    \end{enumerate}

    Furthermore, the branched double cover of $[-1,n+1]\times [1,-n-1]$ can be identified with $D_{n,tw}$ so that the oriented double covers of the edges labelled with $L_i$ in Figure \ref{Fig:Quiver-Graphs} are sent to the appropriate curves in Figure \ref{Fig:Twisted Disk}.
\end{construction}

\begin{proof}
    In Figures \ref{Fig:Pieces}, we have indicated how the double branched cover around the edges locally looks like. By gluing together, the local pieces in Figure \ref{Fig:Pieces}, we obtain the identification of the double branched cover of the square $[-1,n+1]\times [1,-n-1]$ with $D_{n,tw}$ which carries the oriented double covers of the edges labelled $L_i$ to the curves labelled $L_i$ in Figure \ref{Fig:Twisted Disk}. 
    
    If we choose the filling $\mathcal{H}_n$ which contracts the edges labelled with $M_i$ in Figure \ref{Fig:Quiver-Graphs} then we see that $\mathcal{H}_n$ deformation retracts onto the top-half identified with $D_{n,tw}$ whose homology class is generated by the $L_i$. So this filling fulfills all claimed properties.

    \begin{figure}
    \begin{picture}(200,300)
    \put(0,0){\includegraphics[width=8cm]{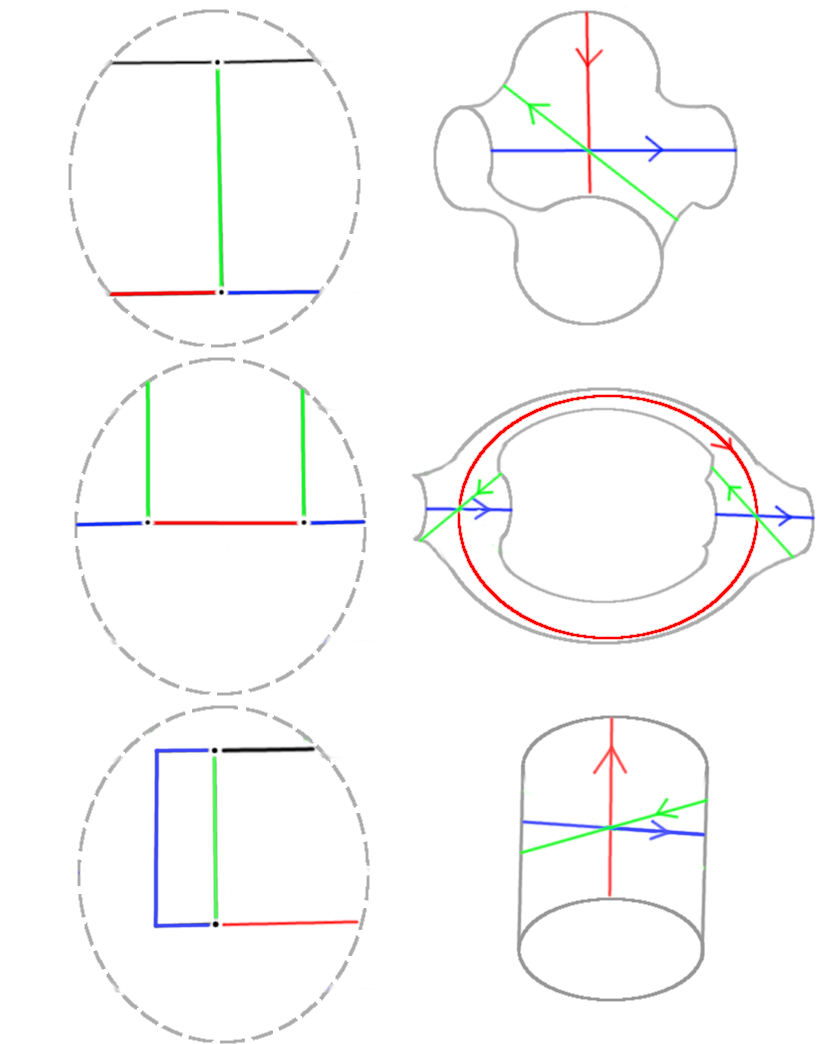}}

    \end{picture}
    \caption{The $3$ local pieces around horizontal edges, vertical edges and bigons. The left picture indicates the local structure of the graph while the right is the corresponding piece of the double branched cover. Colors indicate the oriented double lifts of certain edges. Note that if two edges are given the same colour then they are unrelated if they are not already connected in the left picture.}  
    \label{Fig:Pieces}
\end{figure}
\end{proof}

In the next section, we will construct two sequences of mutations which both start at the graph $\Gamma^1_n$ and end at $\Gamma^2_n$.

\begin{definition}
    Let $\Gamma_0$ be a cubic planar graph with branched lift $\pi$, $\mathcal{C} \subset H_1(\Sigma_g)$ a strictly convex cone. Furthermore, let $E_1,\dots,E_k$ and $E'_1,\dots,E'_l$ be two sequences of simple closed curves on $\Sigma_g$ such that each sequence fulfills the Conditions in Corollary \ref{Mutation sequence corollary}. Denote by $\Gamma$ and $\Gamma'$ the final graph in these sequences. We say that $E_1,\dots,E_k$ and $E'_1,\dots,E'_l$ are mutation equivalent if under the identification of vertices and faces, we have that $\Gamma$ and  $\Gamma'$ agree as graphs and the oriented double lift of their edges are homotopic to one another relative to the branching points.
\end{definition}

\begin{definition}
     Let $\Gamma_0$ be a cubic planar graph with branched lift $\pi$, $\mathcal{C} \subset H_1(\Sigma_g)$ a strictly convex cone. Furthermore, let $E_1,\dots,E_k$ be a sequence which fulfills the properties of Corollary \ref{Mutation sequence corollary}. We say that $E_1, \dots,E_k$ avoids a vertex $v \in \Gamma_0$ if under the appropriate identifications of vertices along the sequence none of the $E_i$ have $v$ as an endpoint. 
\end{definition}

\begin{proposition} \label{Psi's agree}
    Let $\Gamma^1_n$ be the square graph, $\mathcal{H}_n$ the filling from Construction \ref{Construction of filling} and $\mathfrak{p}$ be a choice of points as in Definition \ref{Face relation-definition}. Let $E_1,\dots,E_k$ and $E'_1,\dots,E'_l$ be two sequences of oriented curves on $\Sigma_n$ which are mutation equivalent. Furthermore, assume that both sequences avoid the vertices of $\Gamma^1_n$ which are on the line $(\cdot,n)$. Then there is a choice of oriented $2$-chain which bounds $\mathfrak{p} \times I$ in $\Sigma_n \times I$ which is disjoint from the links in the products $\Psi_{E_1,\dots,E_k}$ and $\Psi_{E'_1,\dots,E'_l}$, as defined in Corollary \ref{Mutation sequence corollary}. Furthermore, the images of $\Psi_{E_1,\dots,E_l}$ and $\Psi_{E'_1,\dots,E'_l}$ under the boundary inclusion $Sk_\mathcal{C}(\Sigma_n \times I) \rightarrow Sk_\mathcal{C}(\mathcal{H}_n)$ agree.
\end{proposition}

\begin{proof}
 Since, we assume that the mutation sequences avoid the vertices on the line $(\cdot,n)$. We may assume that all graphs agree outside the square $[-1,n+1] \times [-n-1,1]$. We may furthermore, assume that all points of $\mathfrak{p}$ are chosen outside this square. We can choose an oriented $2$-chain $D$ as the oriented double lift of a path from the point $p_f \in f$ to one of the vertices on the boundary which avoids the subsquare and thus all edges $E_1,\dots,E_k,E'_1,\dots,E'_l$ and so also avoids all $\mathbb{E}(E_i)$ and all $\mathbb{E}(E'_l)$. Now, fix attention on the left-most bigonal face $f_b$ and the first $n-1$ square faces $f_{1},\dots,f_{n-1}$ of $\Gamma^1_n$. Fix the top-right vertices $v_b,v_1,\dots,v_{n-1}$ of the faces $f_b,f_1,\dots,f_{n-1}$. Then from Corollary \ref{Mutation sequence corollary}, we have the following two sets of equations:

 \begin{eqnarray*}
    \mathcal{A}_{\Gamma_k,f_b,v_b,D}\Psi_{E_1,\dots,E_k} &=& \Psi_{E_1,\dots,E_k}\mathcal{A}_{\Gamma^1_n,f_b,v_b,D} \\
    \mathcal{A}_{\Gamma_l',f_b,v_b,D}\Psi_{E'_1,\dots,E'_l} &=& \Psi_{E'_1,\dots,E'_k}\mathcal{A}_{\Gamma^1_n,f_b,v_b,D}
 \end{eqnarray*}

 for the bigonal face and:

  \begin{eqnarray*}
     \mathcal{A}_{\Gamma_k,f_i,v_i,D}\Psi_{E_1,\dots,E_k} &=& \Psi_{E_1,\dots,E_k}\mathcal{A}_{\Gamma^1_n,f_i,v_i,D} \\
     \mathcal{A}_{\Gamma_l',f_i,v_i,D}\Psi_{E'_1,\dots,E'_l} &=& \Psi_{E'_1,\dots,E'_k}\mathcal{A}_{\Gamma^1_n,f_i,v_i,D} 
 \end{eqnarray*}

for each of the first $n-1$ square faces of $\Gamma^1_n$. The additional $D$ subscript indicates the image under the $D$-untwisting map, see Section \ref{twisted skein module section}. By assumption, the sequences $E_1,\dots,E_k$ and $E'_1,\dots,E'_l$ are mutation equivalent and so the homotopy classes of all edges agree. As $\mathcal{A}_{\Gamma,f,v}$ depends only on the homotopy classes of the edges adjacent to the face, we also have that $\mathcal{A}_{\Gamma_k,f_b,v_b,D}=\mathcal{A}_{\Gamma_l',f_b,v_b,D}$ and $\mathcal{A}_{\Gamma_k,f_i,v_i,D}=\mathcal{A}_{\Gamma_l',f_i,v_i,D}$. So, in particular these equations imply:

\begin{eqnarray} \label{Bigonal face relation}
    \mathcal{A}_{\Gamma^1_n,f_b,v_b,D} \Psi^{-1}_{E_1,\dots,E_k}\Psi_{E'_1,\dots,E'_l}&=& \Psi^{-1}_{E_1,\dots,E_k}\Psi_{E'_1,\dots,E'_l}\mathcal{A}_{\Gamma^1_n,f_b,v_b,D} \\
    \mathcal{A}_{\Gamma^1_n,f_i,v_i,D} \Psi^{-1}_{E_1,\dots,E_k}\Psi_{E'_1,\dots,E'_l}&=& \Psi^{-1}_{E_1,\dots,E_k}\Psi_{E'_1,\dots,E'_l}\mathcal{A}_{\Gamma^1_n,f_i,v_i,D}
\end{eqnarray}

We claim, that these equations imply that this implies the image of $\Psi:=\Psi^{-1}_{E_1,\dots,E_k}\Psi_{E'_1,\dots,E'_l}$ under the boundary inclusion $Sk_\mathcal{C}(\Sigma_n \times I) \rightarrow Sk_\mathcal{C}(\mathcal{H}_n)$ lives inside $Sk_\mathcal{C}(\mathcal{H}_g;D_1,\dots,D_n)$ where $D_1,\dots,D_n$ are disjoint disks which bound the edges labelled by $M_1,\dots,M_n$ in Figure \ref{Fig:Quiver-Graphs}. The complement $\mathcal{H}_n \setminus (\bigcup D_i)$ deforms to a sphere which, in combination with the fact that $Sk(S^3)$ is one-dimensional over the base ring, implies that $\Psi$ agrees with its leading coefficient which is $1$. Thus $\Psi_{E_1,\dots,E_k}=\Psi_{E'_1,\dots,E'_l}$ as claimed.

So, we will start to prove the claim by inductively showing that $\Psi \in Sk_\mathcal{C}(\mathcal{H}_g;D_1,\dots,D_k)$ for all $1 \leq k \leq n$. We start by writing Equation \ref{Bigonal face relation}, explicitly: 

\begin{eqnarray*}
    (a^{-1}[\bigcirc] + (-a)^{\gamma_D} [-M_1])\Psi = \Psi (a^{-1}[\bigcirc] + (-a)^{\gamma_D} [-M_1])
\end{eqnarray*}

where the exponent ${\gamma_D}$, depends on the choice of $D$. We can immediately simplify this as:

\begin{eqnarray*}
    (-a)^{\gamma_D} [-M_1]\Psi = \Psi (-a)^{\gamma_D} [-M_1]
\end{eqnarray*}

and consider its image under the boundary inclusion:

\begin{eqnarray*}
    (-a)^{\gamma_D} [-M_1]\Psi = \Psi (-a)^{\gamma_D} [\bigcirc]
\end{eqnarray*}

where we used that $M_1$ becomes contractible in our choice of filling. Now, we may apply Lemma \ref{Sliding operator lemma} and infer $\Psi \in Sk_{\mathcal{C}}(\mathcal{H}_g;D_1)$ where $D_1$ is a disk filling $M_1$.

So inductively, assume that $\Psi \in Sk_{\mathcal{C}}(\mathcal{H}_g;D_1,\dots,D_k)$ where $D_i$ is a disk filling the edge labelled by $M_i$ in Figure \ref{Fig:Quiver-Graphs}. Then we consider the $k$-th square face equation which we explicitly write out as:

\begin{eqnarray*}
    (a^{-1}[\bigcirc] + (-a)^{\gamma_{1,D}} [l_{v_1}] + (-a)^{\gamma_{2,D}} [l_{v_2}] + (-a)^{\gamma_{D}} [-M_{k+1}])\Psi &= \\ \Psi (a^{-1}[\bigcirc] + (-a)^{\gamma_{1,D}} [l_{v_1}] + (-a)^{\gamma_{2,D}} [l_{v_2}] + (-a)^{\gamma_{D}} [-M_{k+1}]) &
\end{eqnarray*}

where we used the notation of Definition \ref{Face relation-definition}. We observe the following two facts: 

\begin{enumerate}
    \item The homology classes of the inclusions of these elements is $0$ for $[M_{k+1}]$ and $-L_k$ for $[l_{v_1}]$ and $[l_{v_2}]$.
    \item If $\Psi \in Sk_{\mathcal{C}}(\mathcal{H}_g;D_1,\dots,D_k)$ then $\Psi \in Sk_{\widetilde{\mathcal{C}}}(\mathcal{H}_g;D_1,\dots,D_k)$, where 
    
    $\widetilde{\mathcal{C}} := \mathcal{C}\cap \bigcap \{L_i = 0\}$.
\end{enumerate}
The first follows immediately from the choice of basis as written in Figure \ref{Fig:Quiver-Graphs}. The second follows from the fact that $D_1,D_2+D_1,\dots,D_n+\dots+D_1$ define a basis of $H^2(\mathcal{H}_n,\Sigma_n)$ dual to the basis $L_1,\dots,L_n$ of $H_1(\mathcal{H}_n)$. Since all knots are disjoint from the disks $D_1,\dots,D_k$ their oriented intersection number with them is $0$ and thus their $L_1,\dots,L_k$-coefficients are $0$. 

Using these two observations, we may immediately split the equation into two equations one in $Sk_{\widetilde{\mathcal{C}}}(\mathcal{H}_n)$ and one in $Sk_{\widetilde{\mathcal{C}},-|L_k|}(\mathcal{H}_n)$:

\begin{eqnarray*}
     (-a)^{\gamma_{D}} [-M_{k+1}]\Psi &=& \Psi (-a)^{\gamma_{D}} [\bigcirc] \\
      ((-a)^{\gamma_{1,D}} [l_{v_1}] + (-a)^{\gamma_{2,D}} [l_{v_2}])\Psi &=& \Psi((-a)^{\gamma_{1,D}} [l_{v_1}] + (-a)^{\gamma_{2,D}} [l_{v_2}])
\end{eqnarray*}

The first equation allows us to immediately apply Lemma \ref{Sliding operator lemma} and so we conclude $\Psi \in Sk_{\mathcal{C}}(\mathcal{H}_g;D_1,\dots,D_{k+1})$

\end{proof}
     
\subsection{The mutation sequences} \label{Mutation sequences section}

In this section, we will describe the sequence of mutations which will give rise to the left-hand and right-hand sides of Theorem \ref{Main theorem}. We follow the structure, that we will first define the sequence of graphs and then how they are inductively related by positive graph mutations. We only, have to reconstruct one specific sequence for the left-hand side while the right-hand side depends on some initial choices, see Lemma \ref{Quiver-equivalent formulation of sequence}. 

Before we start the discussion, we will fix some notation: In this section, we will only modify the graph $\Gamma^1_n$ in the subsquare $[-1,n+1] \times [-n-1,1]$. There are $n+1$ vertices in this square which lie on the half-lines $\{i\} \times [-n,0]$. Any graph and any mutation below will allow us to identify the vertices occurring along these lines with one another. So, independent of the exact positioning along these lines we will always denote them by $v_i$ for $0 \leq i \leq n$. Similarly, we will always have at most one edge between $v_i$ and $v_j$, we will always denote this edge by $E_{[i+1,j]}$. Similarly, to the notation of $L_{[i+1,j]}$ from Subsection \ref{Subsection:Statement of the main result} this should indicate that $E_{[i+1,j]}$ is an appropriate succession of Dehn twists. As it turns out, this will be true up to a possible change of orientation. However, a priori these two notations are disjoint from one another. In addition, in our graphical calculus, we will indicate that $E_{[i+1,j]}=L_{[i+1,j]}$ if both vertices are at the same height. Otherwise, they agree after reversing orientation.

Lastly, there are $n+2$-many edges $H_{-1},\dots,H_{n+1}$ ordered from left to right. Their intersection with the boundary of the subsquare will always be fixed. This induces an identification of these half-edges in any graphs, we define below.

In the following proofs we will also need to keep track of homotopy classes of different edges. So, it will be convienent to have the following lemma at hand:

\begin{lemma}
    Let $\Sigma_n$ be a genus $n$ surface and $\gamma_1,\dots,\gamma_k$ be a system of simple closed curves such that $\gamma_i \cap \gamma_{j}$ consists of a single point if $|i-j|<2$ and is empty otherwise. Then the set of iterated Dehn-twists $\gamma_{[i,j]}$ fulfills the following properties:

    \begin{eqnarray} \label{Dehn twist identity-left addition}
        \tau_{\gamma_1} (\gamma_{[2,k]}) &=& \gamma_{[1,k]} \\
        \label{Dehn twist identity-right addition}
        \tau_{\gamma_k} (\gamma_{[1,k-1]}) &=& \gamma_{[1,k]} \\ \label{Dehn twist identity- right substraction}
        \tau_{-\gamma_{[j,k]}}(\gamma_{[1,k]}) &=& \gamma_{[1,j-1]} \\ \label{Dehn twist identity- right overspill}
        \tau_{-\gamma_{[1,k]}}(\gamma_{[1,j]}) &=& -\gamma_{[j+1,k]} \\\label{Dehn twist identity- left overspill}
        \tau_{-\gamma_{[2,k]}}(\gamma_{[1,k]}) &=& \gamma_{1}
    \end{eqnarray}
\end{lemma}

\begin{proof}
    We will only verify Equation \ref{Dehn twist identity- right substraction}, the others follow similarly. Observe that if $\gamma_1,\dots,\gamma_k$ is a system of curves as above then $\gamma_{[1,j]}, \gamma_{[j+1,k]}$ can be homotoped in such a way that they only have one transverse intersection point as well. Then, since all smoothing operations are entirely local, it does not matter whether we Dehn twist $\gamma_{[1,j]}$ iteratively along $\gamma_{j+1},\dots,\gamma_{k}$ or Dehn twist $\gamma_{[1,j]}$ along $\gamma_{[j+1,k]}$. Finally, in Figure \ref{Fig:Dehn-twist-identity}, we verify that $\tau_{-\gamma_{[j,k]}}(\gamma_{[1,k]})=\tau_{-\gamma_{[j,k]}}(\tau_{\gamma_{[j,k]}}(\gamma_{[1,k]}))$ agrees with $\gamma_{[1,j-1]}$.

    \begin{figure}
    \begin{picture}(250,100)
    \put(0,0){\includegraphics[width=10cm]{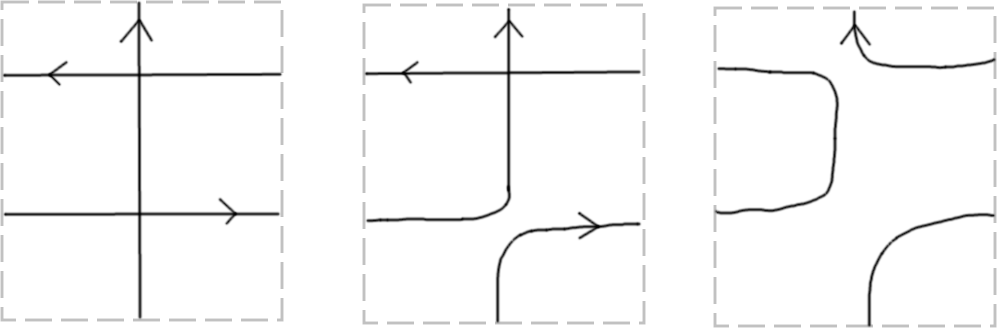}}
    \end{picture}
    \caption{On the right, three curves $\gamma'_1,\gamma_2'$ and $-\gamma'_2$ such that $\gamma_1'$ and $\gamma'_2$ have a single intersection point. In these pictures, we consider the left and the right borders to be glued to one another so that this picture contains the whole curves $\gamma'_2,-\gamma'_2$. The middle picture depicts the Dehn twist of $\gamma'_1$ along $\gamma'_2$. The right picture depicts the iterated Dehn twist of $\gamma'_1$ first along $\gamma'_2$ and then along $-\gamma'_2$. It is evident that the curve in the right picture is homotopic through simple closed curves to $\gamma'_1$.}  
    \label{Fig:Dehn-twist-identity}
\end{figure}
\end{proof}

\subsubsection{The short sequence}
 
\begin{definition} \label{Short sequence definition}
    Let $\Gamma^1_n$ be the square graph. Define the graph $\Gamma_i$, for $1 \leq i \leq n+1$, by modifying $\Gamma^1_n$ in the square $[i-1,n+1] \times [-n,1]$ by:

    \begin{enumerate}
        \item The vertex $v_{i-1}$ is in position $(i-1,-n)$ while $v_j$ for $i\leq j \leq n$ is shifted to $(j,-n+j)$. 
        \item The half-edge $H_{j}$ is attached to the vertex $v_{j+1}$ for $i\leq j \leq n-1$, $H_n$ is attached to $v_n$ and $H_{n+1}$ is attached to $v_{i-1}$.
        \item The homotopy class of the half-edge $H_j$ in $\Gamma_i$ differs from the one in $\Gamma^1_n$ by a Dehn twist along $E_{j+1}$ for $i-1 \leq j \leq n$. The homotopy class of the half-edge $H_{n+1}$ is changed by a Dehn-twist along $E_{[i,n]}$.  
        \item For each $1 \leq j \leq n$, there is an edge labelled by $E_j$.
        \item The edges $E_j$ are homotopic to $-L_{j}$ for $i \leq j \leq n$.
    \end{enumerate}

    For $i<(n+1)$, the graph $\Gamma_i$ inside the square $[i-1,n+1] \times [-n,1]$ is depicted on the top right of Figure \ref{Fig:Quiver-Calculation-1}.
\end{definition}

\begin{construction} \label{Short sequence construction}
    A positive graph mutation along $E_{i}$ in the graph $\Gamma_{i+1}$ induces the graph $\Gamma_i$.
\end{construction}

\begin{proof}
    This is verified in Figure \ref{Fig:Quiver-Calculation-1}. On the left is the special case of the transition $\Gamma_{n+1}=\Gamma^1_n$ (top left) to $\Gamma_n$ (bottom left). On the right we have indicated the general case for $1<i<n$. The other special case $i=1$, is similar to the case $1<i<n+1$ except that the left-most two edges are joint to a half-edge.

    \begin{figure}
    \begin{picture}(350,300)
    \put(0,0){\includegraphics[width=12.5cm]{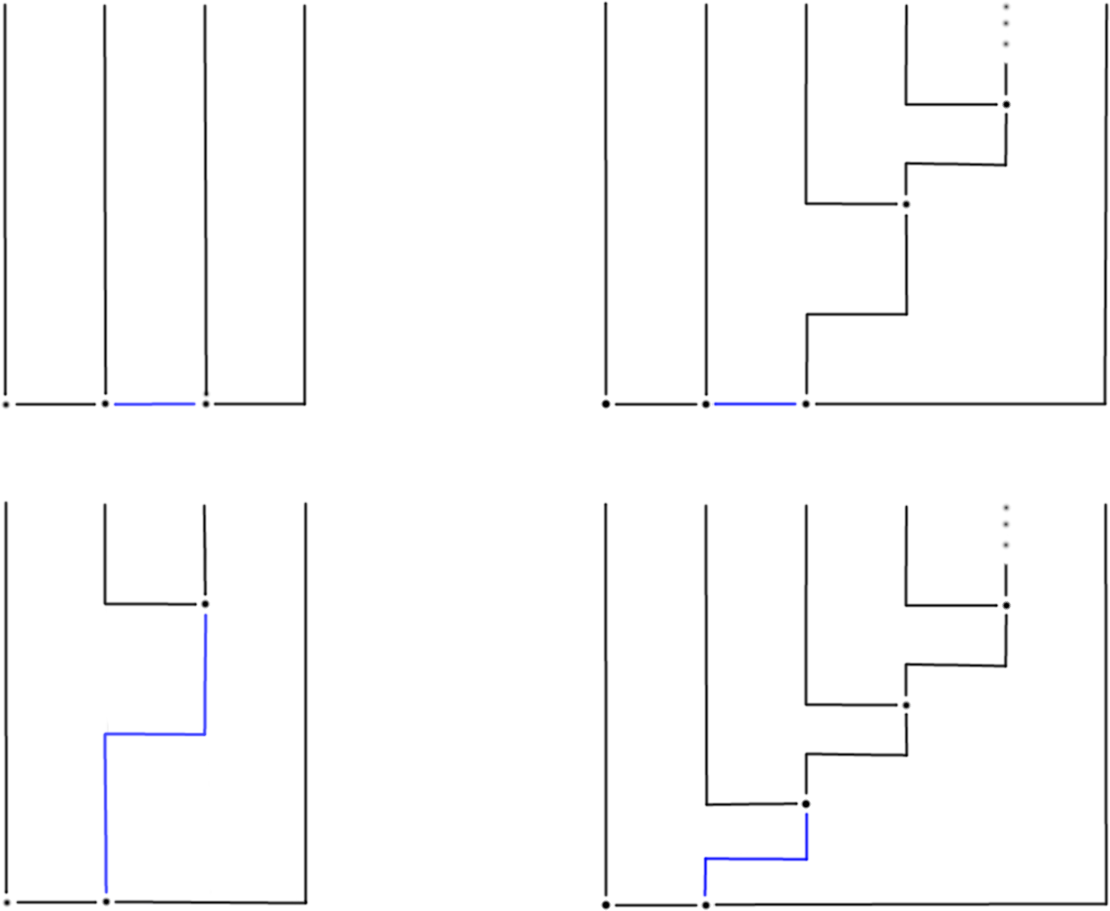}}
    \put(40,153){$L_n$}
    
    \put(40,105){$\tau_{L_n}$}
    \put(40,40){$-L_n$}
    \put(100,40){$\tau_{L_n}$}
    
    \put(230,153){$L_j$}
    \put(300,155){$\tau_{L_{[j+1,n]}}$}

    \put(205,102){$\tau_{L_j}$}
    \put(230,25){$-L_j$}
    \put(300,-5){$\tau_{L_{[j,n]}}$}

    \end{picture}
    \caption{Two graph transitions, the initial graphs are at the top. The edge along we mutate is highlighted in blue. For the homotopy label transition on the right we have used the Dehn-twist identity \ref{Dehn twist identity-left addition}.}  
    \label{Fig:Quiver-Calculation-1}
\end{figure}
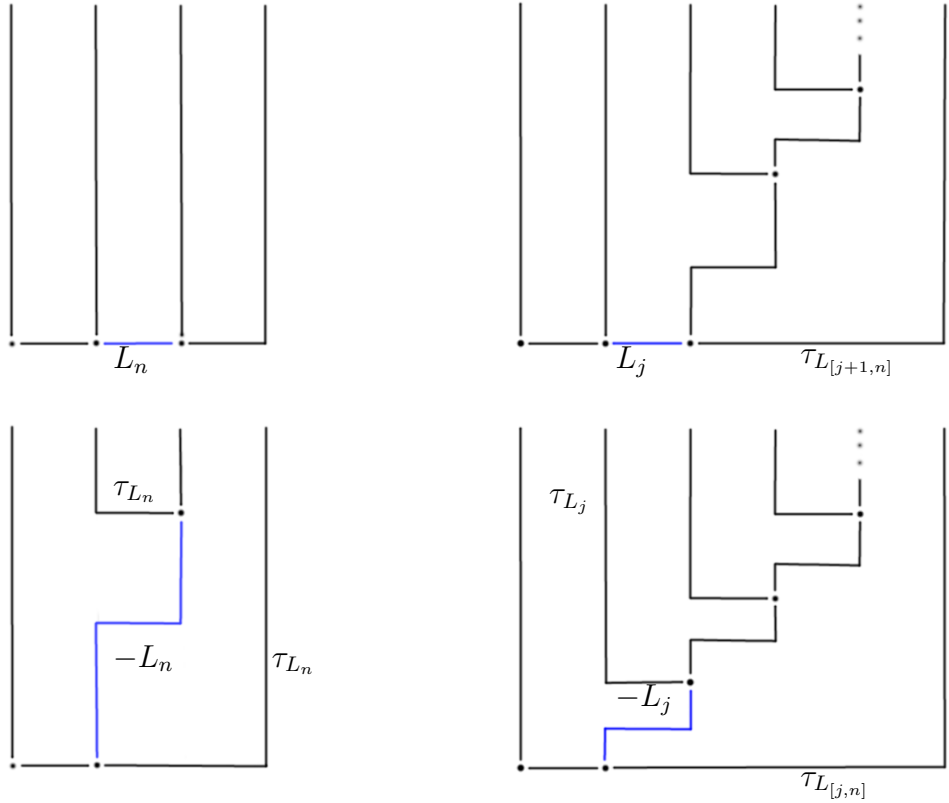

\end{proof}

\subsubsection{The long sequence}

\begin{definition} \label{Long sequence graph definition}
    Let $A=(i_1,\dots,i_n)$ be an $n$-tuple of numbers such that: 
    \begin{enumerate}
        \item $i_j \leq j$ for $1 \leq j \leq n$;
        \item Let $M$ be the maximum such that the first $M$-entries of $A$ agree with $(1,\dots,M)$. Then after deleting the first $M-1$ entries of $A$, it is a non-increasing sequence.
    \end{enumerate}
    Furthermore, for each $l\leq M$ denote by $\mu_{l,min}$ the minimal $j$ with $M\leq j$ such that $i_{\mu_{l,min}} = l$. Define $\mu_{l,max}$ similarly but $\mu_{l,max}$ need not be larger than $M$.
    
    Then, we define $\Gamma_A$ to be the graph which agrees with $\Gamma^1_n$ outside of $[-1,n+1] \times [-n,1]$ and inside:

    \begin{enumerate}
        \item The vertex $v_0$ is in position $(0,-n)$, the vertices $v_k$ are in position $(k,-n+i_k)$.
        \item The half-edges $H_k$ for $-1 \leq k<M$ are connected to $v_{k+1}$. The half-edge $H_M$ is connected to $v_M$. The half-edges $H_k$ for $k>M$ are connected to $v_k$. If $i_n = 0$ then $H_{n+1}$ is connected to $v_n$ and otherwise $H_{n+1}$ is connected to $v_0$.
        \item Under the identification of vertices and faces of $\Gamma_A$ and $\Gamma^1_n$, the homotopy class of $H_{k}$ for $0 \leq k \leq M-1$ is changed by $E_{k+1}$ from the one in $\Gamma^1_n$. The homotopy classes of the remaining half-edges agree with the ones on $\Gamma^1_n$. 
        \item For each $1\leq l \leq M$ let $v_{l-1},v_{\mu_{l,min}}, \dots v_{\mu_{l,max}}$ be the vertices at height $l-1$. Then there are edges $E_{[i+1,j]}$ between vertices $v_i,v_j$ which are neighbors on this list. If this list contains only $v_{l-1}$ then there are no such edges.
        \item For each $0 \leq l < M$ there is an edge $E_{[l+1,\mu_{l+1,max}]}$ between $v_{l}$ and $v_{\mu_{(l+1),max}}$.
        \item If the edge $E_{[i,j]}$ exists then under the identification between $\Gamma^1_n$ and $\Gamma_n$, the homotopy class of $E_{[i,j]}$ and $L_{[i,j]}$ agree if both vertices are on the same horizontal line $[-1,n+1] \times \{y\}$. Otherwise, its homotopy class is $-L_{[i,j]}$.
    \end{enumerate}
\end{definition}

\begin{construction} \label{Long sequence construction}
    Let $A=(i_1,\dots,i_n)$ be an $n$-tuple fulfilling the properties listed in Definition \ref{Long sequence graph definition} and $\Gamma_A$ the associated graph. Furthermore, let $E_{[i_j,j]} \in \Gamma_A$ be an edge such that $E_{[i_j,j]}=L_{[i_j,j]}$ under the appropriate identification of $\Gamma^1_n$ and $\Gamma_A$. Then positive mutation of $\Gamma_A$ along $E_{[i_j,j]}$ yields the graph $\Gamma_{A'}$ where $A$ and $A'$ agree except in the $j$-th position where the $j$-th value of $A'$ is $i_j+1$. 
\end{construction}

\begin{proof}
    There are $8$ different local pictures of the graph to consider:

    \begin{enumerate}
        \item $j=M+1, i_j=M$ and $i_{j+1}=M$, in this case the local calculation is done on the left of Figure \ref{Fig:Quiver-Calculation-2};
        \item $j=M+1, i_j=M$ and $i_{j+1}<M$, the local calculation is done on the right of Figure \ref{Fig:Quiver-Calculation-2};
        \item $j>M+1$, $i_{j+1}=i_j$ and $i_{j-1}=i_j+1$. Then this calculation is done on the left of Figure \ref{Fig:Quiver-Calculation-3}, after inserting $l=j$.
        \item $j>M+1$, $i_{j+1}=i_j$ and $i_{j-1}>i_j+1$. Then this calculation is done on the left of Figure \ref{Fig:Quiver-Calculation-3}, after inserting $l=i_j$.
        \item $j>M+1$, $i_{j+1}>i_j$ and $i_{j-1}=i_j+1$. Then this calculation is done on the right of Figure \ref{Fig:Quiver-Calculation-3}, after inserting $l=j$.
        \item $j>M+1$, $i_{j+1}>i_j$ and $i_{j-1}>i_j+1$. Then this calculation is done on the right of Figure \ref{Fig:Quiver-Calculation-3}, after inserting $l=i_j$.
        \item $j=n$, $i_{n}=0$ and $i_{n-1}>1$. This case is similar to the case depicted on the left of Figure \ref{Fig:Quiver-Calculation-3} except that the right vertex is replaced with the half-edge $H_{n+1}$.
        \item $j=n$ and $i_{n}=0$. This case is similar to the case depicted on the left of Figure \ref{Fig:Quiver-Calculation-3} except that the right vertex is replaced with the half-edge $H_{n+1}$ and $l=i_{n}$ if $i_{n-1}=1$ and $l=1$ otherwise. 
        
    \end{enumerate}
    
\end{proof}

     \begin{figure}
    \begin{picture}(320,225)
    \put(0,0){\includegraphics[width=8cm]{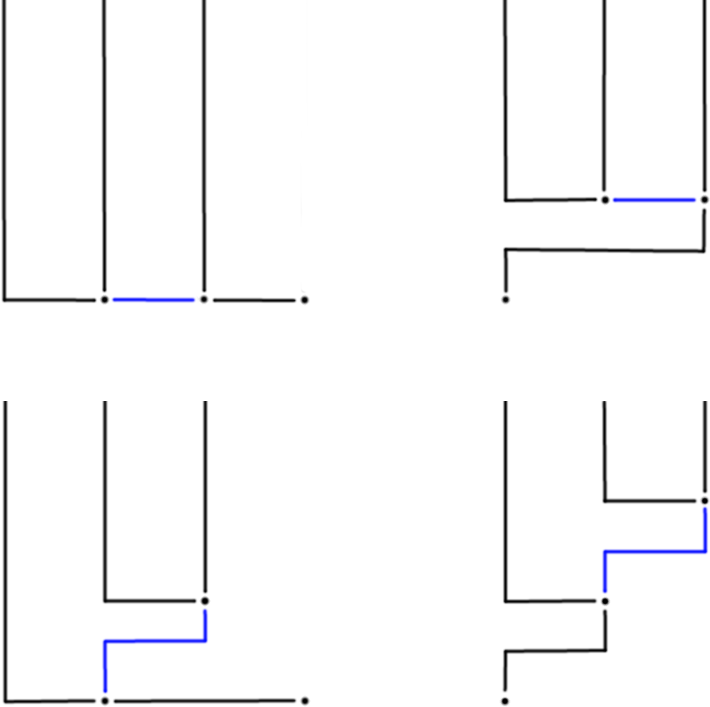}}
    
    \put(48,120){$L_j$}
    \put(82,120){$L_{j+1}$}
    
    \put(42,85){$\tau_{L_j}$}
    \put(40,10){$-L_j$}
    \put(50,-10){$L_{[j,j+1]}$}
    
    \put(200,170){$L_j$}
    \put(180,135){$-L_{[j-1,j]}$}

    \put(200,80){$\tau_{L_j}$}
    \put(230,50){$-L_j$}
    \put(180,5){$-L_{j-1}$}

    \end{picture}
    \caption{Two graph transitions, the initial graphs are at the top. The edge along we mutate is highlighted in blue. For the homotopy label transition on the left, we have used the Dehn-twist identity \ref{Dehn twist identity-left addition}. On the right, we have used the Dehn-twist identity \ref{Dehn twist identity- right substraction}.}  
    \label{Fig:Quiver-Calculation-2}
\end{figure}
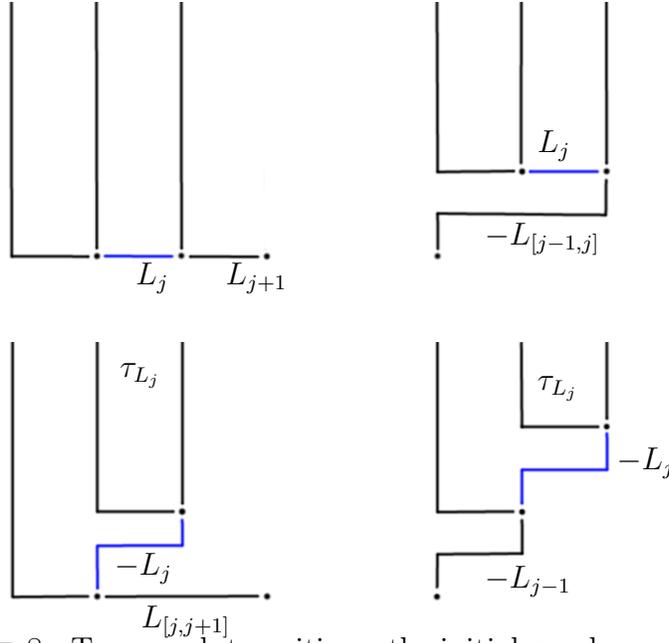

     \begin{figure}
    \begin{picture}(350,200)
    \put(0,0){\includegraphics[width=10cm]{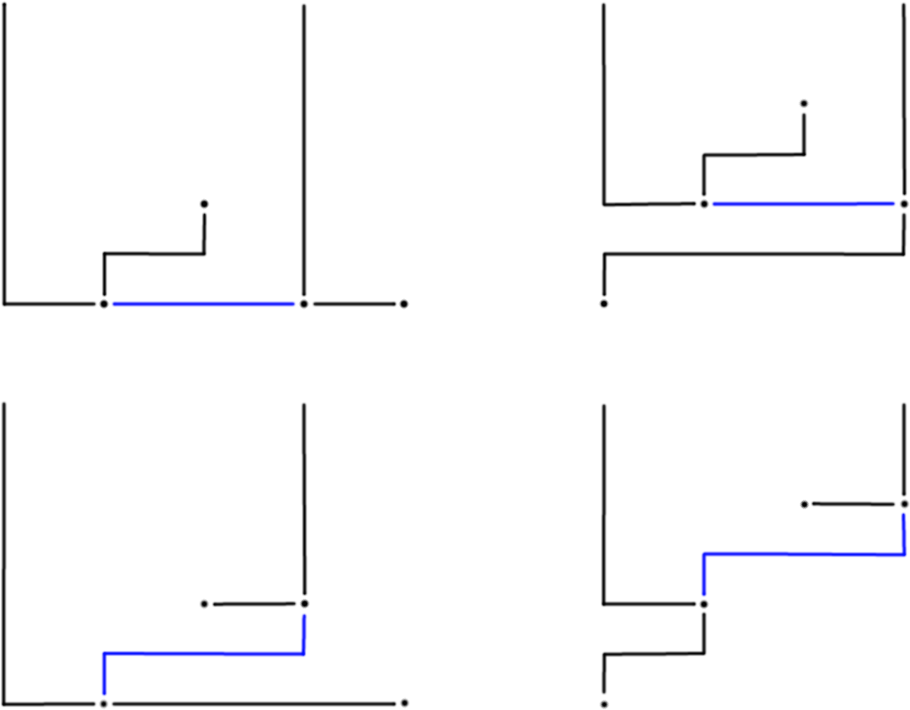}}
      
    \put(50,117){$L_{[i_j,j]}$}
    \put(100,117){$L_{j+1}$}
    \put(20,148){$-L_{[i_j,l]}$}
    
    \put(55,40){$L_{[l+1,j]}$}
    \put(45,22){$-L_{[i_j,j]}$}
    \put(70,-8){$L_{[i_j,j+1]}$}
    
    \put(235,148){$L_{[i_j,j]}$}
    \put(205,180){$-L_{[i_j,l]}$}
    \put(205,130){$-L_{[i_j-1,j]}$}

    \put(255,75){$L_{[l+1,j]}$}
    \put(220,40){$-L_{[i_j,j]}$}
    \put(190,5){$-L_{i_j-1}$}

    \end{picture}
    \caption{Two graph transitions, the initial graphs are at the top. The edge along we mutate is highlighted in blue. For the homotopy label transition on the left we have used the Dehn-twist identities  \ref{Dehn twist identity-right addition} and \ref{Dehn twist identity- right overspill}. On the right, we have used the Dehn-twist identities \ref{Dehn twist identity- right overspill} and \ref{Dehn twist identity- left overspill}}  
    \label{Fig:Quiver-Calculation-3}
\end{figure}
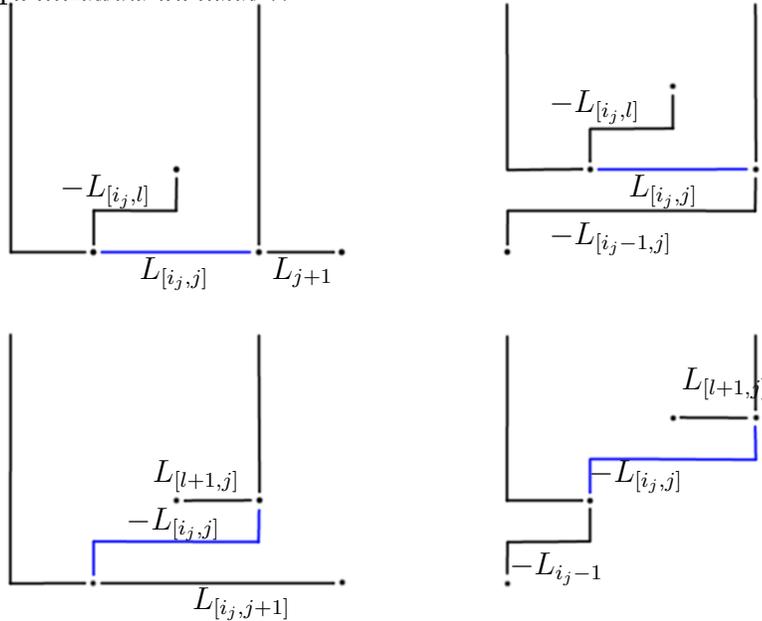

\section{Proof of Theorem \ref{Main theorem}}
This section is dedicated to proving the main theorem which we restate below for convenience: 

\begin{theorem}(Skein-valued lift of the $Q$-relations)
    Let $D_{n,tw}$ be the twisted n-holed disk from Figure \ref{Fig:Twisted Disk} and $L_1,\dots,L_n$ be the curves indicated there. Furthermore, fix an order $\alpha_1<\dots<\alpha_N$ on $[1,1],[1,2],\dots,[1,n],\dots,[n,n]$ such that the irreducible representations $V_{\alpha_i}$ obey that $V_{\alpha_j}<V_{\alpha_i}$ if $Hom_Q(V_{\alpha_i},V_{\alpha_j}) \neq 0$. Then the following relation holds in $Sk_\mathcal{C}(\mathcal{H}_n)$:

    \begin{equation} \label{Main equation}
        \mathbb{E}(L_{1})\dots \mathbb{E}(L_{n}) = \mathbb{E}(L_{\alpha_N})\dots\mathbb{E}(L_{\alpha_1})
    \end{equation}

    where $\mathcal{C}$ is the cone spanned by the $L_1,\dots,L_n$.
\end{theorem}

\begin{proof}
    Consider the square graph $\Gamma^1_n$ and the lift of its edges to $\Sigma_n$ as in Construction \ref{Construction of filling}. There it was shown that the branched lift of the subsquare $\Box:=[-1,n+1]\times [-n-1,1]$ and the edges completely contained therein lead to an identification with $D_{n,tw}$ and the set of curves $L_1,\dots,L_n$, as in Figure \ref{Fig:Twisted Disk}. Furthermore, we fix a filling $\mathcal{H}_n$ of $\Sigma_n$ as in Construction \ref{Construction of filling} which contracts the edges labelled with $M_i$'s in Figure \ref{Fig:Quiver-Graphs}. This allows us to identify $\mathcal{H}_n$ with $D_{n,tw} \times I$. We observe, as in Subsection \ref{Section 3.1.}, that the boundary inclusion $\Sigma_n \times I \rightarrow \mathcal{H}_n$ restricts to the left stacking action for links which live inside $D_{n,tw} \times I \subset \Sigma_n \times I$. This defines our algebra structure on $Sk_\mathcal{C}(\mathcal{H}_n)$ and the $Sk_\mathcal{C}(\Sigma_n)$-action on $Sk_\mathcal{C}(\mathcal{H}_n)$.

    Next, we define the left-hand side of Equation \ref{Main equation}: Direct inspection yields that $\Gamma_n^1$ and $\Gamma_{n+1}$, from Definition \ref{Short sequence definition}, coincide, after identifying the appropriate vertices. See the beginning of Subsection \ref{Mutation sequences section} for the discussion of these identifications. Then by Construction \ref{Short sequence construction}, we may iteratively apply positive mutations along the edges $E_{n},\dots,E_1$ to define a sequence of positive graph mutations $\Gamma_{n+1},\dots,\Gamma_1$ with $\Gamma_1=\Gamma^2_n$ after identifying the appropriate vertices. All the edges appearing here live in the subsquare $\Box$ and thus so do the elements $\mathbb{E}(L_{n}),\dots, \mathbb{E}(L_{1})$. This also implies that this sequence of mutations avoid all vertices not contained in this square. We set $\Psi_L:=\mathbb{E}(L_{1})\dots\mathbb{E}(L_{n})$.

    We turn to the right-hand side: Here, we need to consider the order $\alpha_1,\dots,\alpha_N$ of the intervals $[1,1],[1,2],\dots[1,n],\dots,[n,n]$ that we chose initially. Denote by $a_0,\dots,a_n$ the associated sequence from Lemma \ref{Quiver-equivalent formulation of sequence}. We observe that $\Gamma_{a_0}=\Gamma^1_n$, since $a_0=(0,\dots,0)$. Now, by Property $4$ of Lemma \ref{Quiver-equivalent formulation of sequence}, we know that the interval $\alpha_{k+1}=[i,j]$ must be adapted to $A_k$ in such a way that after deleting the first $M_k$ elements of $A_k$, $j$ is the beginning position for some constant subsequence. However, this translates to the fact that the graph $\Gamma_{A_k}$ contains some edge $E_{\alpha_{k+1}}$ homotopic to $E_{\alpha_{k+1}}$. By Construction \ref{Long sequence construction}, we obtain the graph $\Gamma_{A_{k+1}}$ after positive mutation along this edge. So, iteratively we can positively mutate along a sequence of edges $E_{\alpha_1},\dots,E_{\alpha_N}$ which are homotopic to $L_{\alpha_1},\dots,L_{\alpha_N}$ which induce the sequence of graphs $\Gamma_{a_0},\dots,\Gamma_{a_N}$. One readily observes that $a_N=(1,\dots,n)$ and that $\Gamma_{(1,\dots,n)}=\Gamma^2_n$ after identifying the appropriate vertices. Again all the edges along which we mutate live in the subsquare $\Box$ and thus do the elements $\mathbb{E}(L_{\alpha_1}),\dots,\mathbb{E}(L_{\alpha_N})$. Again, this implies that this sequence of mutations avoids all vertices which are not contained in this square and we set $\Psi_R:= \mathbb{E}(L_{\alpha_N})\dots\mathbb{E}(L_{\alpha_1})$.

    The sequence $L_n,\dots,L_1$ and $L_{\alpha_1},\dots,L_{\alpha_N}$ are mutation equivalent as both sequences start with the graph $\Gamma^1_n$ and end with the graph $\Gamma^2_n$ under the appropriate identification of vertices. In addition, both sequences avoid the vertices on the line $(\cdot,n)$ and so by Proposition \ref{Psi's agree} $\Psi_L=\Psi_R$ after including along $Sk_\mathcal{C}(\Sigma_n \times I) \rightarrow Sk_\mathcal{C}(\mathcal{H}_n)$. Furthermore, as all elements live in $D_{n,tw} \times I$ the boundary inclusion restricts to the stacking operation.
\end{proof}

\bibliographystyle{hplain}
\bibliography{skeinrefs}

\begin{thebibliography}{10}

\bibitem{unknot}
Tobias Ekholm and Vivek Shende.
\newblock Skein recursion for holomorphic curves and invariants of the unknot.
\newblock {\em arXiv:2012.15366}.

\bibitem{Faddeev1993}
Ludwig Faddeev and A~Yu Volkov.
\newblock Abelian current algebra and the {V}irasoro algebra on the lattice.
\newblock {\em Physics Letters B}, 315(3-4):311--318, 1993.

\bibitem{Faddeev1994}
Ludwig~D Faddeev and Rinat~M Kashaev.
\newblock Quantum dilogarithm.
\newblock {\em Modern Physics Letters A}, 9(05):427--434, 1994.

\bibitem{Gilmer-Zhong-connectsum}
Patrick~M Gilmer and Jianyuan~K Zhong.
\newblock The {HOMFLYPT} skein module of a connected sum of 3--manifolds.
\newblock {\em Algebraic \& Geometric Topology}, 1(1):605--625, 2001.

\bibitem{Hu}
Mingyuan Hu.
\newblock A proof of the pentagon relation for skeins.
\newblock {\em arXiv:2401.10817}.

\bibitem{HSZ}
Mingyuan Hu, Gus Schrader, and Eric Zaslow.
\newblock Skeins, clusters and wavefunctions.
\newblock {\em arXiv:2312.10186}.

\bibitem{Keller}
Bernhard Keller.
\newblock On cluster theory and quantum dilogarithm identities.
\newblock {\em Representations of algebras and related topics}, 5:85, 2011.

\bibitem{Nakamura}
Lukas Nakamura.
\newblock Recursion relations and bps-expansions in the {HOMFLY}-{PT} skein of the solid torus.
\newblock {\em arXiv:2401.10730}.

\bibitem{Reineke}
Markus Reineke.
\newblock Poisson automorphisms and quiver moduli.
\newblock 2009, 0804.3214.

\bibitem{Scharitzer-Shende}
Matthias Scharitzer and Vivek Shende.
\newblock Quantum mirrors of cubic planar graph {L}egendrians.
\newblock {\em arXiv:2304.01872}.

\bibitem{Scharitzer-Shende-2}
Matthias Scharitzer and Vivek Shende.
\newblock Skein valued cluster transformation in enumerative geometry of {L}egendrian mutation.
\newblock {\em arXiv:2312.10625}.

\bibitem{schutzenberger}
Marcel-Paul Schutzenberger.
\newblock Une interpretation de certaines solutions de lequation fonctionnelle-{F} (x+ y)= {F} (x) {F} (y).
\newblock {\em COMPTES RENDUS HEBDOMADAIRES DES SEANCES DE L ACADEMIE DES SCIENCES}, 236(4):352--354, 1953.

\bibitem{Treumann-Zaslow}
David Treumann and Eric Zaslow.
\newblock Cubic planar graphs and {L}egendrian surface theory.
\newblock {\em arXiv:1609.04892}.

\end{thebibliography}

\end{document}